\documentclass[a4paper]{amsart}
\usepackage{color}
\usepackage{amsfonts,amsmath, amssymb,amsthm}

\renewcommand{\ge}{\geqslant}
\renewcommand{\le}{\leqslant}

\usepackage{isomath}
\usepackage{lmodern}

\usepackage{setspace}
\usepackage{graphicx,float,enumerate,subfigure,tikz}
\usepackage[ruled,boxed,commentsnumbered,norelsize]{algorithm2e}
\usepackage{verbatim}
\usepackage{url}
\usepackage{epstopdf}
\usepackage[capitalise]{cleveref}

\newtheorem{theorem}{Theorem}
\newtheorem{corollary}[theorem]{Corollary}
\newtheorem{lemma}[theorem]{Lemma}
\newtheorem{proposition}[theorem]{Proposition}

\newtheorem{conjecture}[theorem]{Conjecture}

\newtheorem*{theorem*}{Theorem}
\newtheorem*{proposition*}{Proposition}

\theoremstyle{definition}

\theoremstyle{remark}
\newtheorem{remark}[theorem]{Remark}

\crefname{remark}{Remark}{Remarks}

\theoremstyle{definition}
\newtheorem{claim}{Claim}
\usepackage{etoolbox}

\usepackage{mathtools}

\DeclareMathOperator{\aff}{aff}

\DeclareMathOperator{\ver}{vert}

\crefname{rmk}{Remark}{Remarks}
\crefname{problem}{Problem}{Problems}
\date{\today}
\title{Polytopes close to being simple}

\author{Guillermo Pineda-Villavicencio}
\address{Centre for Informatics and Applied Optimisation, Federation University Australia}
\email{\texttt{work@guillermo.com.au}}
\author{Julien Ugon}

\address{School of Information Technology, Deakin University}
\email{\texttt{julien.ugon@deakin.edu.au}}

\author{David Yost}
\address{Centre for Informatics and Applied Optimisation, Federation University Australia}
\email{\texttt{d.yost@federation.edu.au}}


\begin{document}
\begin{abstract}
It is known that polytopes with at most two nonsimple vertices are reconstructible from their graphs, and that $d$-polytopes with at most $d-2$ nonsimple vertices are reconstructible from their 2-skeletons. Here we close the gap between 2 and $d-2$, showing that certain polytopes with more than two nonsimple vertices are reconstructible from their graphs. In particular, we
prove that reconstructibility from graphs also holds for $d$-polytopes  with $d+k$ vertices and at most  $d-k+3$ nonsimple vertices, provided $k\ge 5$. For $k\le4$, the same conclusion holds under a slightly stronger assumption.

Another measure of deviation from simplicity is the {\it excess degree} of a polytope, defined as $\xi(P):=2f_1-df_0$, where $f_k$ denotes the number of $k$-dimensional faces of the polytope. Simple polytopes are  those with excess zero. We prove that polytopes with excess at most $d-1$ are reconstructible from their graphs, and this is best possible.
An interesting intermediate result is that $d$-polytopes with less than $2d$ vertices, and at most $d-1$ nonsimple vertices, are necessarily pyramids.
\end{abstract}

\maketitle

\section{Introduction and summary}

The $k$-dimensional {\it skeleton} of a polytope $P$ is the set of all its faces of dimension $\le k$. The 1-skeleton of $P$ is the {\it graph} $G(P)$ of $P$. Reconstructing  a polytope from its $k$-skeleton amounts to giving the combinatorial structure of the polytope (i.e the  lattice of its faces, ordered by inclusion)  solely by querying the $k$-skeleton. It however suffices to reconstruct the facets of $P$, since the combinatorial structure of  a polytope is determined by the vertex-facet incidence graph, where a facet is adjacent to a vertex if and only if, it contains the vertex \cite[Sec.~16.1.1]{GooORo04}. Throughout the paper, we let $d$ denote the dimension of $P$, $\deg$ denote the {\it degree} of a vertex, i.e. the number of edges incident to the vertex in the polytope $P$, and  $V(P)$ and $E(P)$ denote the vertex and edge set of a polytope $P$, respectively.

Every $d$-polytope is reconstructible from its $(d-2)$-skeleton \cite[Thm.~12.3.1]{Gru03}, and there are combinatorially inequivalent $d$-polytopes with the same $(d-3)$-skeleton: take, for instance, a bipyramid over a $(d-1)$-simplex and a pyramid over a bipyramid over a $(d-2)$-simplex.
For some classes of polytopes, the graph somewhat surprisingly determines the combinatorial structure of the polytope: polytopes with dimension at most three and simple polytopes \cite{BliMan87,Kal88}. A $d$-polytope is  called {\it simple} if every vertex is simple. A vertex is called {\it simple} if its degree is exactly  $d$; otherwise it is called {\it nonsimple}. Equivalently, a vertex is simple if it is contained in exactly $d$ facets, and nonsimple otherwise.

Define the {\it excess degree} $\xi(u)$ of a vertex $u$ in a $d$-polytope as $\deg u-d$. Then the {\it excess} $\xi$ of a $d$-polytope is defined as the sum of the excess of all its vertices; i.e.~ \[\xi(P):=\sum_{u\in \ver P} (\deg u-d).\]
Simple polytopes have excess zero.
Polytopes with small excess are a natural generalisation of simple polytopes. The excess degree is studied in detail in \cite{PinUgoYos16a}. For the early sections of  this paper, we will just need the following basic but surprisingly useful  result \cite[Lem. 2.4, 2.5 and 2.6(i)]{PinUgoYos16a}.

\begin{lemma}\label{lem:basic-excess} Let $P$ be a $d$-polytope, $F$ a facet, and let $v$ be a vertex in $F$.
\begin{enumerate}[(i)]
\item Suppose $v$ is a simple vertex in $F$, but is not simple in $P$. Then there is facet $J$ containing $v$ whose intersection with $F$ is not a ridge.
\item Let $J$ be any facet which is distinct from $F$, such that  $F \cap  J$ is not a ridge. Then every vertex in $F \cap  J$ is nonsimple in $P$.
\item Suppose $v$ is nonsimple in $P$, and adjacent to a simple vertex $w$ of $P$ in $P\setminus F$. Then $v$ must be adjacent to another vertex in $P\setminus F$, other than $w$.
\end{enumerate}
\end{lemma}

It was shown in \cite{NevPinUgo17}  that polytopes with  at most two nonsimple vertices are reconstructible from their graphs. Since it will be used several times, we state this explicitly here.

\begin{theorem}\label{thm:fourpointfive} \cite[Thms.~4.5 and 4.8]{NevPinUgo17} Every polytope with only one or two nonsimple vertices can be
reconstructed from its graph.
\end{theorem}

On the other hand \cite{NevPinUgo17} also exhibited a pair of inequivalent 4-polytopes $Q_4^1$ and $Q_4^2$ with eight  vertices, three of them nonsimple in each case, and with the same graph. In particular, they  are not reconstructible from their common graph, and \cref{thm:fourpointfive} does not extend to polytopes with more than two nonsimple vertices. Their construction is described in detail in \cite[Sec.~2]{NevPinUgo17}; here we simply illustrate them (\cref{fig:dpoly2dvertd-1ns} (b-c)). There are eight facets in  $Q_4^1$ but only seven in $Q_4^2$. One of the facets of  $Q_4^2$ is a  bipyramid over a simplex, namely 02467; its missing 2-face 246 is highlighted in \cref{fig:dpoly2dvertd-1ns}(c). This bipyramid is
split into two simplices to form  $Q_4^1$ (\cref{fig:dpoly2dvertd-1ns}(b)). In fact, this construction extends to higher dimensions; we summarise  further information about it in \cref{rmk:polytopes-Q_{q}} and refer to \cite[Prop.~2.2]{NevPinUgo17} for the details. Other reconstruction results can be found in \cite[Sec.~20.5]{GooORo04}.

\begin{figure}
\begin{center}
\includegraphics{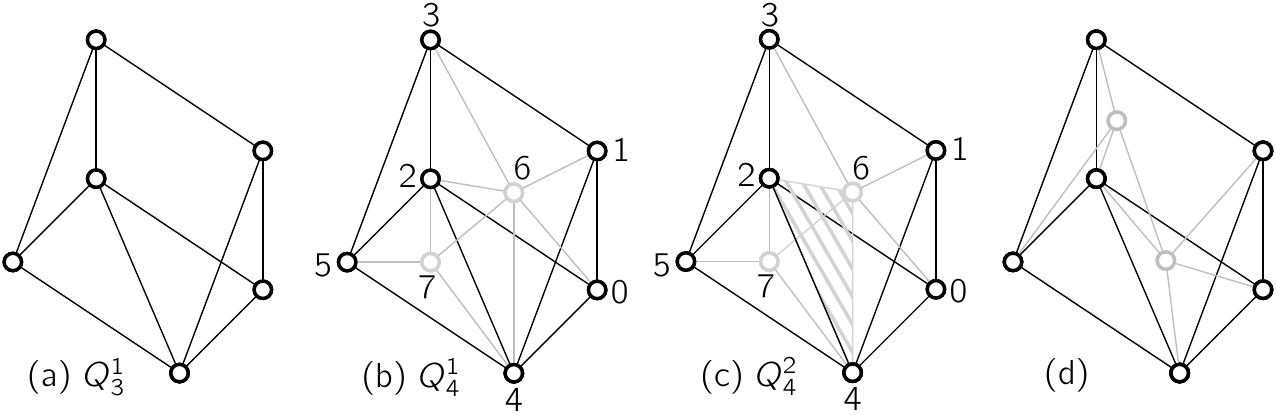}
\end{center}
\caption{Examples of $d$-polytopes with $2d$ vertices, of which precisely $d-1$ are nonsimple. (a) The polytope $Q_3^1$. (b-c) The pair of 4-polytopes $Q^{1}_{4}$ and $Q^2_{4}$. The missing 2-face of the bipyramid face 02467 in $Q^{2}_{4}$ is highlighted. (d) Polytope obtained from a simplicial $4$-prism by ``pulling''  a vertex (cf.~\cite[Sec.~5.2]{Gru03}) along one of the edges of a simplex facet.}
\label{fig:dpoly2dvertd-1ns}
\end{figure}

\begin{remark}\label{rmk:polytopes-Q_{q}}
For $d\ge 4$, there are inequivalent $d$-polytopes $Q^1_d$ and $Q^2_d$, each with $2d$ vertices, exactly $d-1$ of them nonsimple, and the same $(d-3)$-skeleton. The polytope $Q_d^1$ contains exactly $2d$ facets, while the polytope $Q^2_d$ contains $2d-1$ facets. The polytope $Q^2_d$ is obtained from $Q^1_d$ by gluing two simplex facets of $Q^{1}_{d}$ along a common ridge $R$ to create the bipyramid over $R$, which becomes a facet of $Q_{d}^{2}$.

For $d=3$, the construction in \cite{NevPinUgo17}  leads to two inequivalent 3-polytopes $Q^1_3$ and $Q^2_3$, each with six vertices, and hence the same 0-skeleton.  The polytope $Q_3^1$ (\cref{fig:dpoly2dvertd-1ns}(a)), sometimes called the tetragonal antiwedge, contains exactly six facets and two  nonsimple vertices, but the polytope $Q^2_3$, the triangular prism, contains only five facets and is actually simple.
\end{remark}

Here we keep studying the structure and reconstruction of polytopes which are ``close'' to being simple.   We consider two approaches which generalise the concept of simplicity and guarantee reconstructibility from graphs.

\begin{description}
\item[Approach 1] Consider $d$-polytopes with ``few'' nonsimple vertices; this is the approach taken in \cite{NevPinUgo17}.
\item[Approach 2] Consider $d$-polytopes with small excess.
\end{description}

Concerning Approach 1, the polytopes $Q^1_d$ and $Q^2_d$ (cf.~\cref{rmk:polytopes-Q_{q}}) show that in general by ``few'', we must mean at most $d-2$ nonsimple vertices. With regard to Approach 2, the aforementioned  pair of a bipyramid over a $(d-1)$-simplex and a pyramid over a bipyramid over a $(d-2)$-simplex have excess exactly $d$. So by small excess we must mean excess at most $d-1$. But then the excess theorem \cite[Thm.~3.4]{PinUgoYos16a} states that the smallest values of the excess of a $d$-polytope are 0 and $d-2$.

The main results of this paper are summarised next. We want to highlight that some of our results here came about after the authors tested a number of hypotheses on  \texttt{polymake} \cite{GawJos00}.

\begin{table}
\begin{tabular}{c c c c c}
{Facet}&{Polytope 1}&{Polytope 2}&{Polytope 3}&Polytope 4\\
\hline
 0:&\{2 3 4 5 6\}&\{2 3 4 5 6\}&\{1 2 3 4 5 6\}&\{2 3 4 5 6\}\\
1:&\{1 3 4 5 6\}&\{1 3 4 5 6\}&\{0 3 4 5 6\}&\{1 3 4 5 6\}\\
2:&\{1 2 5 6\}&\{0 1 2 5 6\}&\{0 2 5 6\}&\{0 1 2 5 6\}\\
3:&\{1 2 4 6\}&\{1 2 4 6\}&\{0 2 4 6\}&\{0 1 2 3 4\}\\
4:&\{1 2 3 5\}&\{0 2 3 5\}&\{0 2 3 5\}&\{1 2 4 6\}\\
5:&\{0 2 3 4\}&\{0 1 3 5\}&\{0 1 3 4\}&\{0 2 3 5\}\\
6:&\{0 1 3 4\}&\{1 2 3 4\}&\{0 1 2 4\}&\{0 1 3 5\}\\
7:&\{0 1 2 4\}&\{0 1 2 3\}&\{0 1 2 3\}\\
8:&\{0 1 2 3\}\\
\hline
\end{tabular}
\caption{Vertex-facet incidences of all nonpyramidal 4-polytopes with seven vertices and  four nonsimple vertices. They all have the same graph, and degree sequence $(4,4,4,5,5,6,6)$. They can be obtained from the catalogues  in \cite{FukMiyMor13a,Gru03}.}
\label{tab:4Polytopes7Vertices4Nonsimple}
\end{table}

In \S3, we apply Approach 1. A polytope with only $d+1$ vertices is obviously a simplex. The structure of polytopes with $d+2$ vertices is well understood \cite[Sec. 6.1]{Gru03}. They have either $d-2,d$ or $d+2$ nonsimple vertices. If such a polytope has only $d-2$ nonsimple vertices, it must be a $(d-2)$-fold pyramid over a quadrilateral, and so is reconstructible from its graph. We have already noted distinct examples with $d$ nonsimple vertices but the same graph.

For $d$-polytopes with $d+3$ vertices, we have a new result: if such a polytope has at most $d-1$ nonsimple vertices, then it is a  $(d-3)$-fold  pyramid over one of just three 3-dimensional examples, and consequently, the graph  determines its entire combinatorial structure. This is best possible in the sense that there are nonpyramidal  $d$-polytopes with $d+3$ vertices and exactly $d$ nonsimple vertices which are not reconstructible from their graphs (cf.~\cref{tab:4Polytopes7Vertices4Nonsimple}).
	
For a $d$-polytope $P$ with $d+4$ vertices, the slightly stronger assumption that $P$ has at most $d-2$ nonsimple vertices, is enough to ensure that the graph of $P$ determines its entire combinatorial structure. Furthermore, in the case of $P$ having $d-1$ nonsimple vertices, the polytope is still reconstructible from its 2-skeleton. Again these results are best possible, as shown by the examples $Q^1_4$ and $Q^2_4$.
	
Then for $k\ge 5$, any $d$-polytope with $d+k$ vertices and at most $d-k+3$ nonsimple vertices is determined by its graph of $P$.  In the particular case $k=5$, the pair of polytopes $Q^1_5$  and $Q^2_5$ shows that this is best possible.

In view of these three results and the results of \cite{NevPinUgo17} we venture to conjecture the following.

\begin{conjecture}\label{conj:smallVertRec}Let $P$ be a $d$-polytope with at most $d-2$ nonsimple vertices. Then the graph of $P$ determines its entire combinatorial structure. \end{conjecture}

These results depend on some results about pyramids, which are presented first in \S2. Recall a polytope is an {\it $r$-fold pyramid} if it is a pyramid whose {\it basis} is an $(r-1)$-fold pyramid, and any polytope is a 0-fold pyramid. If a vertex $u$ is an apex of a pyramid $P$, we say that $P$ is {\it pyramidal} at $u$. The main conclusion of \S2, of interest in its own right, is that a $d$-polytope with at most $d+k$ vertices and at most $d-1$ nonsimple vertices, is necessarily a $(d-k)$-fold pyramid. This is only informative if $k<d$. For $k=d$, we have the following modification: a $d$-polytope with $2d$ vertices, and at most $d-2$ nonsimple vertices, is either a simplicial $d$-prism or a pyramid. Furthermore, this is best possible as there are 4-polytopes with eight vertices and three nonsimple vertices which are neither simplicial 4-prisms nor pyramids (namely $Q_4^1$ and $Q_4^2$). Recall a {\it simplicial $d$-prism}  is any prism whose base is a $(d-1)$-simplex.

In \S4, we pay regard to polytopes with small excess, and completely settle the reconstruction problem for them by proving that for a $d$-polytope with excess at most $d-1$, the graph determines its entire combinatorial structure. This result is best possible in the sense that there are $d$-polytopes with excess $d$ which are not reconstructible from their graphs.

\section{Polytopes with a small number of nonsimple vertices are Pyramids}

Here we show that knowing a polytope has strictly less than $d$ nonsimple vertices gives us a lot of information about its structure. In particular, large classes of such polytopes must be pyramids. This is crucial for the reconstruction results in the next section.
Let us call two vertices of a polytope {\it nonneighbours} if they are not adjacent. Note that every vertex is thus a nonneighbour of itself.

\begin{theorem} \label{thm:pyramidRec}
Let $P$ be a $d$-polytope,  which contains at most $d-k$ nonsimple vertices, where $k\le d$. Then either $P$ is a pyramid, or each nonsimple vertex has at least $k$ simple nonneighbours.

In case $P$ is pyramidal,  it is, for each $j$, reconstructible from its $j$-skeleton if and only if the basis is reconstructible from its $j$-skeleton.
\end{theorem}

\begin{proof}
The second alternative in the conclusion is trivially true if $k=0$.  If $k=d$, then $P$ is simple, and the second alternative in the conclusion is vacuously true. Henceforth, we assume that $0<k<d$.

Consider first the case that some nonsimple vertex $u$ is nonadjacent to at most $k-1$ simple vertices.

Removing all the nonsimple vertices and all the simple vertices which are not adjacent to $u$ cannot disconnect the graph, according to Balinski's theorem \cite{Bal61}. Therefore, the graph $G(S)$ induced by the set $S$ of simple vertices which are neighbours of $u$ is connected. Let $x$ be one such simple vertex in $S$. Then, there is a facet $F$ containing $x$ but not $u$.  Given any simple vertex in $V(F)\cap S$, all its neighbours other than $u$ must also be in $F$. Since $G(S)$ is connected, $F$ contains all the  vertices in $S$.  If some vertex $y\ne u$ is not in this facet, it cannot be a neighbour of any member of $S$. But outside $S$ there are at most $d-k+k-1$ vertices, including $y$; this mean that $y$ has  degree at most $d-2$ in the polytope. This absurdity implies that every vertex of $P$ is in $F\cup\{u\}$, i.e. $P$ is a pyramid with basis $F$ and apex $u$.

We only prove the reconstruction statement for graphs (1-skeletons), but the result extends to $j$-skeletons for $j\ge 2$.

So suppose now that $P$ is pyramidal with basis $F$ and apex $u$. If $F$ is reconstructible from its graph, then we can obtain the vertex set of each  $(d-2)$-face $R$ of $F$, and from it, the corresponding facet of $P$, the one with vertex set $V(R)\cup \{u\}$. Thus, we can get the vertex-facet incidence graph of $P$. Otherwise $P$ is not reconstructible.
\end{proof}

There are examples of nonsimplicial pyramidal and simplicial nonpyramidal $d$-polytopes with exactly $d$ nonsimple vertices and the same graph; look no further than our old friends, the bipyramid over a $(d-1)$-simplex and the pyramid over a bipyramid over a $(d-2)$-simplex. Thus we get  the following  as corollary of \cref{thm:pyramidRec}.

\begin{corollary}\label{cor:pyramidRec}
A $d$-polytope having at most $d-1$ nonsimple vertices, and a vertex adjacent to every other vertex, must be a pyramid.

Furthermore, this statement is best possible, as there are pyramidal and nonpyramidal $d$-polytopes with exactly $d$ nonsimple vertices, at least one of which is adjacent to every other vertex; and in fact with the same graph.
\end{corollary}

Before proceeding with our results, we need two simple lemmas.

\begin{lemma}[{\cite[Lem.~10(iii)]{PrzYos16}}]\label{lem:simplePolytopes}  Up to combinatorial equivalence, the $d$-simplex and the simplicial $d$-prism are the only  simple $d$-polytopes with no more than $2d$ vertices.
\end{lemma}

\begin{lemma}\label{lem:NoSimpleFaces} Suppose $P$ is a $d$-polytope with $2d$ or fewer vertices, and that some facet $F$ of $P$ contains only  simple vertices. Then $P$ is either a simplicial prism, or a pyramid over $F$.
\end{lemma}

\begin{proof} Recall that the {\it Minkowski sum} of two polytopes $Q+R$ is defined simply as $\{x+y: x\in Q, y\in R\}$,
and that a polytope is called {\it (Minkowski) decomposable} if
it can be written as the Minkowski sum of two polytopes,  which are not similar to it. Actually, these definitions are not really important to us now; more important are the following two results about decomposability.

Shephard  \cite[Thm.~(15)]{She63} proved that if some facet $F$ of a polytope $P$ contains only  simple vertices, then $P$ is either decomposable, or a pyramid over $F$; see \cite[Prop. 5]{PrzYos16} for another proof. And according to \cite[Thm. 9]{PrzYos16}, the only decomposable polytope with $2d$ or fewer vertices is the prism. The lemma follows from combining these two results.\end{proof}

By an application of \cref{lem:simplePolytopes}, $P$ must actually be a triplex as defined in \cite{PinUgoYos15}. We do not need this stronger conclusion here.

Additional assumptions about the total number of vertices now allow us to draw a stronger conclusion.

\begin{theorem}\label{thm:pyramid} Let $P$ be a $d$-polytope with fewer than $2d$ vertices, of which at most $d-1$ are nonsimple. Then $P$ is a  pyramid.
\end{theorem}
\begin{proof}

We proceed by induction on $d$.  The base case $d=2$ is easily proved: indeed, $P$ must be a triangle.   Now assume that the claim is true for dimensions $2,\ldots,d-1$.

If $P$ has a facet with $2d-2$ vertices, there will only be one vertex outside that facet, which ensures that $P$ is a pyramid. So assume that every facet has at most $2d-3$ vertices.

The case of $P$ having a facet in which every vertex is simple is settled by \cref{lem:NoSimpleFaces}. Henceforth  assume also that every facet has at least one nonsimple vertex.

Amongst all the facets of $P$ which omit at least one nonsimple vertex, choose one, say $F_1$, with a maximum number of vertices. The induction hypothesis ensures that $F_1$ is a pyramid over some ridge $R$, say with apex $u_1$. If $R$ is not a simplex, then  $u_1$ is nonsimple in $F_1$.
If $R$ is a simplex, then so is $F_1$, and we may choose any nonsimple vertex to be its apex $u_1$. Then $R$ contains at most $d-3$ nonsimple vertices, because there are at most $d-2$ nonsimple vertices in $F_1$.  In particular, there are simple vertices in $R$.

Let $F_{2}$ be the other facet  containing $R$. Clearly $F_{2}$ omits the nonsimple vertex $u_1$. By the induction hypothesis and the maximality of $F_{1}$, $F_{2}$  is also a pyramid over $R$. Let $u_2$ denote the apex of $F_{2}$.  Any simple vertex in $R$ has all of its neighbours in $F_1\cup F_2$. Suppose that there is a vertex $z$ outside $F_1\cup F_2$. Then removing  $u_1, u_2$ and the nonsimple vertices in $R$, at most $d-1$ vertices altogether,  would disconnect $z$ from the simple vertices in $R$, violating Balinski's theorem.  This  ensures that every vertex of $P$ lies in  $F_1\cup F_2$.

Since there are only two vertices outside the ridge $R$, $P$ is a 2-fold pyramid over $R$.
\end{proof}

Repeated application gives us the following corollary. The case $k=2$ is essentially known, following from the characterisation of $d$-polytopes with $d+2$ vertices  \cite[Sec.~6.1]{Gru03}.

\begin{proposition}\label{cor:d+k} Suppose $1\le k\le d$, and that $P$ is a $d$-polytope with  $d+k$ vertices, of which  at most $d-1$ are nonsimple. Then $P$ is a $(d-k)$-fold pyramid over a $k$-polytope with $2k$ vertices.
\end{proposition}

This begs the question of  $d$-polytopes with $2d$ vertices.
The next result covers that case, and is in the same spirit as \cref{cor:d+k}.

\begin{proposition}
\label{prop:d+d} Let $P$ be a $d$-polytope with $2d$ vertices and at most $d-2$ nonsimple vertices. Then $P$ is either a simplicial $d$-prism or a pyramid.

Furthermore, this is best possible as there are 4-polytopes with eight vertices and three nonsimple vertices which are neither simplicial 4-prisms nor pyramids (namely ~Polytopes $Q_4^1$ and $Q_4^2$).
\end{proposition}

\begin{proof} The  idea used for \cref{thm:pyramid} also proves this proposition, but we give the full details. As in the proof of \cref{thm:pyramid}, we proceed by induction on $d$.  In the  base case $d=2$ we have that $P$ is a quadrilateral, which is a simplicial prism in two dimensions.   Now assume that the claim is true for dimensions $2,3,\ldots,d-1$.

If a facet of $P$ has $2d-1$ vertices, then $P$ is clearly a pyramid. If $P$ is a simple polytope, then it is a simplicial $d$-prism by \cref{lem:simplePolytopes}. We may henceforth assume that every facet has at most $2(d-1)$ vertices, and that some vertices are not simple.

The case when some facet contains only simple vertices is again taken care of by \cref{lem:NoSimpleFaces}. This leaves us with the case that every facet has at least one nonsimple vertex.

There are facets which omit at least one nonsimple vertex. Amongst all such facets, choose one, say $F_1$, with the maximal number of vertices. Then $F_1$ has at most $2(d-1)$ vertices, of which at most $(d-2)-1$ are nonsimple. Then \cref{thm:pyramid} and the induction hypothesis together ensure that $F_1$ is either a  prism or a pyramid.

Suppose $F_1$ is a prism,  and denote by $v$ a vertex in $F_1$ which is not simple in $P$. Clearly every vertex in $F_1$ is simple in $F_1$. Then ~\cref{lem:basic-excess}(i) ensures that there is another facet containing $v$, say $J$, which does not intersect $F_1$ in a ridge. There are only two vertices outside $F_1$, so  $J$ must intersect $F_1$  in a subridge. But then every vertex in  $F_1\cap J$  will be nonsimple in $P$, and so $F_1\cap J$ would contain at least $d-2$ nonsimple vertices. In particular, $F_1$ contains every nonsimple vertex in $P$. This being contrary to the hypothesis that $F_1$ omits a nonsimple vertex, we conclude that $F_1$ is  a pyramid.

We claim that the apex $u_1$ of this pyramid is, or can be chosen to be, nonsimple in $P$. If $F_{1}$  is a simplex, we choose $u_{1}$  to be a nonsimple vertex of $P$ in $F_1$, and define $R$ as the convex hull of the other vertices. Then $R$ is a ridge, and $F_1$ is a pyramid over $R$, whose apex $u_1$ is nonsimple. If $F_1$ is not a simplex, we recall that it is a pyramid over some base $R$, necessarily a ridge, say with apex $u_1$. Since  $R$ will not be a simplex in this case, $u_{1}$ is automatically nonsimple in $F_{1}$, and thus in $P$. Let $F_2$ denote the other facet  containing  $R$.

Then $F_{2}$ omits the nonsimple vertex $u_{1}$, so by maximality, it is also a  pyramid over $R$, say with apex $u_2$. Consequently, if there were a vertex outside $F_{1}\cup F_{2}$, then removing the vertices $u_1, u_2$ and the nonsimple vertices in $R$, at most $d-2$ vertices altogether,  would disconnect the graph of $P$, contradicting Balinski's theorem. Hence there are no vertices outside $F_{1}\cup F_{2}$ and again $P$ is a 2-fold pyramid over $R$. \end{proof}

\cref{prop:d+d} gives a new proof of Gr\"unbaum's result that there is no 4-polytope with eight vertices and 17 edges; see \cite[Thm. 10.4.2, p. 193]{Gru03}. Indeed, such a 4-polytope must have at most two nonsimple vertices, in which case the polytope would be a pyramid. But this is impossible, as the base would have seven vertices and only ten edges.

It seems unlikely that there is any  extension of these results to more than $2d$ vertices. The following question might seem to be natural:
\begin{quote} Must every $d$-polytope with $2d+1$ vertices, of which at most $d-2$ are nonsimple, be either a pentasm or pyramid?
\end{quote}
A $d$-dimensional  {\it pentasm} is the Minkowski sum of a $d$-simplex and a line segment which is parallel to one triangular face, but not parallel to any edge, of the simplex; or any polytope combinatorially equivalent to it. Pentasms were first defined in \cite{PinUgoYos15} and studied further in \cite{PinUgoYos16a}; the graph in \cref{fig:3poly7vert2ns}(b) is that of a 3-dimensional pentasm.

However even this modest question has a negative answer. A counterexample when $d=3$ is given by the graph in \cref{fig:3poly7vert2ns}(a). Two counterexamples in four dimensions are the duals of the polytopes whose Gale diagrams are $10^{th}$ and $17^{th}$ in the list \cite[Fig. 6.3.4]{Gru03}, which are depicted  in \cref{fig:Gale-Close-To-Simple} in that order.

\begin{figure}
\begin{center}
\includegraphics{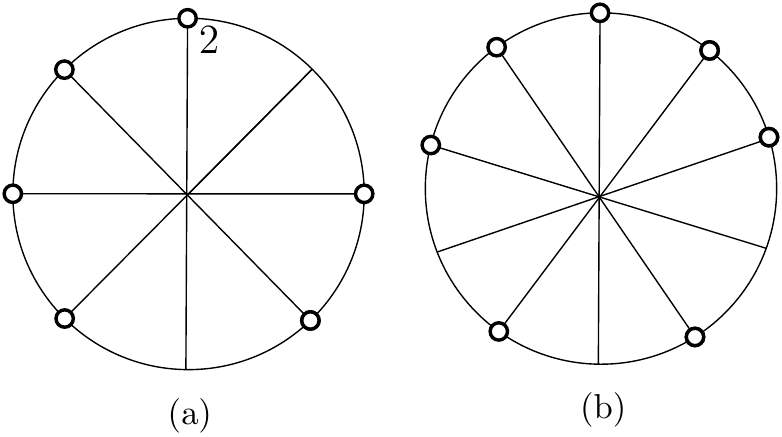}
\end{center}
\caption{Gale diagrams of two nonpyramidal 4-polytopes whose duals have nine vertices, at most two nonsimple vertices and are neither a pentasm or a pyramid.}\label{fig:Gale-Close-To-Simple}
\end{figure}

\section{Reconstruction: Polytopes with a small number of vertices}

Now we apply the preceding work to obtain structural and reconstruction results for polytopes with  less than $2d$ vertices, most of which are simple.
	
\begin{theorem}\label{thm:d+kVert}
Let $k\ge 5$ and let $P$ be a $d$-polytope with $d+k$ vertices, of which at most $d-k+3$  are nonsimple. Then the graph of $P$ determines the entire combinatorial structure of $P$. For the particular case of $k=5$, this conclusion is best possible as shown by the pair of polytopes $Q^1_5$  and $Q^2_5$.
\end{theorem}
	
\begin{proof} If $d< k$ then the polytope is reconstructible from its graph by \cref{thm:fourpointfive}, and the results in \cite{BliMan87,Kal88}. In the case of $d=k$ and $P$ being nonsimple, since $d-k+3\le d-2$, \cref{prop:d+d} gives that $P$ is a pyramid, and the reconstruction follows from  \cref{thm:pyramidRec}  and \cref{thm:fourpointfive} since the basis of the pyramid would have at most two nonsimple vertices. So assume $d\ge k+1$. From \cref{cor:d+k} it ensues that $P$ is a $(d-k)$-fold pyramid over some $k$-polytope $Q$ with $2k$ vertices. Since there are at most  $d-k+3$  nonsimple vertices in $P$, $Q$ has at most three nonsimple vertices. By \cref{thm:pyramidRec} the reconstruction statement now reduces to proving that $Q$ is reconstructible from its graph. By \cref{prop:d+d}, $Q$ is either a simplicial $k$-prism, which  is clearly reconstructible, or else a pyramid over a $(k-1)$-polytope with $2k-1$ vertices, at most two of which are nonsimple, in which case $Q$ is also reconstructible  by combining \cref{thm:pyramidRec} and \cref{thm:fourpointfive}.
\end{proof}

The preceding theorem is not valid for polytopes with $d+4$ or fewer vertices. However reconstructibility holds under  a slightly stronger hypothesis about the number of nonsimple vertices.

\begin{figure}
\begin{center}
\includegraphics{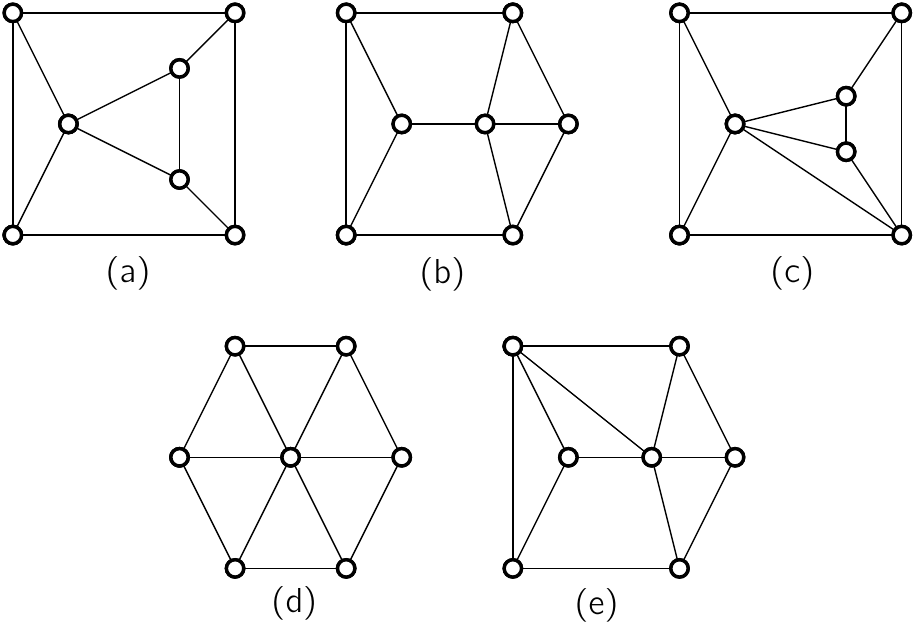}
\end{center}
\caption{Graphs of the 3-polytopes with seven vertices and at most two nonsimple vertices (cf.~\cite[Fig.~4]{BriDun73}).}
\label{fig:3poly7vert2ns}
\end{figure}

\begin{theorem}\label{thm:dPlus4VertRec} Let $P$ be a $d$-polytope with $d+4$ vertices, of which at most $d-1$ are nonsimple. If  $P$ has at most $d-2$ nonsimple vertices, then the graph of $P$ determines its entire combinatorial structure. Furthermore, in the case of $P$ having  $d-1$ nonsimple vertices, the polytope is still reconstructible from its 2-skeleton.

More precisely, for $d=3$ there are exactly five such polytopes (see ~\cref{fig:3poly7vert2ns}). For every $d\ge4$ there are exactly nine such polytopes, namely a $(d-4)$-fold pyramid over either a simplicial 4-prism one of the three polytopes in ~\cref{fig:dpoly2dvertd-1ns} (b-d), or a $(d-3)$-fold pyramid over one of the five 3-polytopes in \cref{fig:3poly7vert2ns}.

These results are best possible, in the sense that there are  nonpyramidal 4-polytopes with eight vertices and three nonsimple vertices which are not reconstructible from their graphs,  namely the ~polytopes $Q^1_4$ and $Q^2_4$ (cf.~\cref{rmk:polytopes-Q_{q}}).
\end{theorem}
	
\begin{proof} If $d\le 3$, $P$ is of course reconstructible from its graph. For $d=3$ the 3-polytopes with seven vertices and at most two nonsimple vertices can be found in \cite[Fig.~4]{BriDun73} or \cref{fig:3poly7vert2ns}.

Graph reconstructibility also holds for simple polytopes. So suppose $P$ is not  a simple polytope and $d\ge4$. Then by \cref{cor:d+k}, $P$ is a $(d-4)$-fold pyramid over a 4-polytope $Q$ with eight vertices.  By \cref{thm:pyramidRec} the reconstruction statement now reduces to proving that $Q$ is reconstructible from its graph, and  this follows from \cref{thm:fourpointfive} and \cite{BliMan87,Kal88}, in case $P$ has at most $d-2$ nonsimple vertices.

For $d\ge 4$, once we have that the polytope is a $(d-4)$-fold pyramid over a 4-polytope with eight vertices and at most three nonsimple vertices, we first look at the  catalogues \cite{FukMiyMor13a} of 4-polytopes with eight vertices to find those with exactly three nonsimple vertices, all of which are depicted in \cref{fig:dpoly2dvertd-1ns}(b-d). Thanks to \cref{prop:d+d}, a 4-polytope with eight vertices and at most two nonsimple vertices is either a simplicial 4-prism or a pyramid over a 3-polytope with seven vertices and at most two nonsimple vertices (see ~\cref{fig:3poly7vert2ns});  and in all these cases it is reconstructible from its graph.

Finally, note that reconstructing from the 2-skeleton reduces to showing that a 4-polytope with eight vertices and exactly three nonsimple vertices is reconstructible from its 2-skeleton, which is a very special case of \cite[Thm.~12.3.1]{Gru03}.
\end{proof}

Finally, we come to the case of $d+3$ vertices. Our original proof of the next result used Gale diagrams, but the following argument seems to be neater.

\begin{theorem}\label{thm:dPlus3VertRec}	Let $P$ be a $d$-polytope with $d+3$ vertices, of which at most $d-1$ are nonsimple. Then the polytope is either a  $(d-3)$-fold  pyramid over  a simplicial 3-prism,  a $(d-3)$-fold  pyramid over $Q^1_3$ (\cref{fig:dpoly2dvertd-1ns}(a)),  or a $(d-2)$-fold pyramid over a pentagon. As a consequence, the graph of $P$ determines its entire combinatorial structure.

These results are best possible in the sense that for any $d\ge4$, there are nonpyramidal  $d$-polytopes with $d+3$ vertices and exactly $d$ nonsimple vertices which are not reconstructible from their graphs.
\end{theorem}

\begin{proof} Let $P$ be a $d$-polytope with $d+3$ vertices and  at most $d-1$ nonsimple vertices. If $P$ is simple, then either $d=3$ and $P$ is a simplicial prism, or $d=2$ and $P$ is a pentagon.
So suppose $P$ is not a simple polytope. \cref{cor:d+k} gives that $P$ is a $(d-3)$-fold pyramid over a 3-polytope, which must be reconstructible from its graph. Hence  $P$ is reconstructible by repeated application of \cref{thm:pyramidRec}.

Once we know that $P$ is a $(d-3)$-fold pyramid, we can obtain all such polytopes simply by looking for 3-polytopes with six vertices, at most two of which are nonsimple. They are the simplicial 3-prism, $Q^1_3$ (\cref{fig:dpoly2dvertd-1ns}(a)), and the pyramid over a pentagon, see \cite[Fig.~3]{BriDun73} or \cite[Fig.~6.3.1]{Gru03}.

For examples of nonpyramidal 4-polytopes with $d+3$ vertices and $d$ nonsimple, see \cref{tab:4Polytopes7Vertices4Nonsimple}. Constructing multifold pyramids over these gives higher dimensional examples.
\end{proof}

In fact, there are $3d-8$ distinct combinatorial types of $d$-polytopes with $d$ nonsimple vertices and three simple vertices, and they have the same graph, namely the complete graph on $d+3$ vertices with a path of length four removed. However, for any $d$, there is a $d$-polytope with $d+1$ nonsimple vertices and two simple vertices which is reconstructible from its graph. We will study this in more detail elsewhere.

\section{Polytopes with small excess} Recall that the excess $\xi$ of a $d$-polytope $P$ is $\xi(P)=\sum_{u\in \ver P} (\deg u-d)$. Polytopes with small excess $\xi\le d-1$ were first studied in \cite{PinUgoYos16a}, where the excess theorem was established.

\begin{theorem}[Excess theorem, {\cite[Thm.~3.3]{PinUgoYos16a}}]\label{thm:excess} The smallest values of the excess of a $d$-polytope are $0$ and $d-2$.
\end{theorem}

In this section we show that, like simple polytopes (those with excess zero) \cite{BliMan87}, all polytopes with small excess are reconstructible from their graphs. It is known that a polytope with dimension at most three is reconstructible from its graph. And there are pairs of $d$-polytopes with excess $d$ and isomorphic $(d-3)$-skeleta: a bipyramid over a $(d-1)$-simplex and a pyramid over a bipyramid over a $(d-2)$-simplex. So by virtue of the excess theorem, we concentrate on polytopes with excess $d-2$ and $d-1$, for $d\ge 4$.

Our capstone result, \cref{thm:ExcessRec}, asserts that any $d$-polytope with excess less than $d$ is reconstructible from its graph. Before delving into its proof, we  recall some definitions and results from \cite{PinUgoYos16a,NevPinUgo17,Jos00}.

\begin{lemma}[Structure of $d$-polytopes with excess $d-2$, {\cite[Lem.~4.8, Thm.~4.10]{PinUgoYos16a}}]\label{lem:excess-d-2-full-story} Every $d$-polytope $P$  with excess exactly $d-2$ has either

\begin{enumerate}[(i)]
\item a unique nonsimple vertex; or
\item  exactly $d-2$ nonsimple vertices, each of degree $d+1$ in $P$, which form a simplex $(d-3)$-face $K$.
\item In the latter case, every facet in $P$ intersecting $K$, but not containing it, misses exactly one vertex of $K$ and every vertex of $K$ in the facet has degree $d$.
\end{enumerate}
\end{lemma}

\begin{lemma}[Structure of $d$-polytopes with excess $d-1$, {\cite[Thm.~4.18]{PinUgoYos16a}}]\label{lem:excess-d-1-full-story} Let $P$ be $d$-polytope with  excess degree $d-1$, where  $d>3$. Then $d=5$ and either

\begin{enumerate}[(i)]
\item there is a single vertex with degree nine; or

\item there are two vertices with degree seven; or

\item there are four vertices each with degree six, which form a quadrilateral 2-face $Q$ which is the intersection of two facets.
Furthermore, every facet in $P$ intersecting $Q$ but not containing it intersects $Q$ at an edge, and every vertex of $Q$ in such a facet has degree five.
\end{enumerate}
\end{lemma}

\begin{proof} Items (i), (ii), and the first sentence of (iii) are restatements of \cite[Thm.~4.18]{PinUgoYos16a}. Here we prove the second part of (iii).

Recall from \cref{lem:basic-excess}(ii) that if two facets of $P$ do not intersect in a ridge, then every vertex in their intersection is nonsimple; and if this intersection is either a vertex or an edge,  then every vertex in their intersection has degree at least seven. So in case (iii), every pair of facets intersects in either a ridge, $Q$, or the empty set.

Let $F_{1}$ and $F_{2}$ be two facets whose intersection is $Q$. Fix  a vertex $u\in Q$. Let $F$ be a facet containing $u$ but missing some neighbour $v$ of $u$ in $Q$. The intersection of $F$ and $F_{i}$ must be a ridge for each $i$. Thus, all the neighbours of $u$ in $P$ except $v$ must be in $F$: there are two such facets.

Furthermore, any facet containing  $u$ and its two neighbours in $Q$ must contain $Q$. This completes the proof of the lemma.
\end{proof}

Our methodology to establish the reconstruction of polytopes with small excess relies on a result of  Joswig \cite{Jos00}, which in turn builds on Kalai's idea to prove the reconstructibility of simple polytopes; see  \cite{Kal88}.

Define a {\it $k$-frame} as a subgraph of $G(P)$ isomorphic to the star $K_{1,k}$, where the vertex of degree $k$ is called the {\it root} of the frame. If the root of a frame is a simple vertex, we say that the frame is {\it simple}.  We say that a $k$-frame with root $x$ is {\it valid} if there is a facet containing $x$ and all the edges of the frame. If $x$ is a simple vertex, each of its $(d-1)$-frames is valid.

\begin{lemma}[{ \cite[Thm.~2.3]{Jos00}}]\label{lem:Joswig-valid-frames}
A polytope can be reconstructed from its graph if the valid frames of each vertex are known.
\end{lemma}

Call an acyclic orientation  of the graph $G(P)$ of a polytope $P$ {\it good } if for every nonempty face $F$ of $P$ the graph $G(F)$ of $F$ has a unique sink.  (A {\it sink} as usual means a vertex with no directed edges going out.) As in \cite[Sec.~4]{NevPinUgo17}, we only need that the acyclic orientation has a unique sink in every facet, so for us this possibly larger set represents the good orientations. An acyclic orientation of $G(P)$  induces a partial ordering of the vertices  of $G(P)$.

Define an {\it initial} set of a graph $G(P)$ with respect to some orientation as a set such that no edge is directed from a vertex not in the set to a vertex in the set. Similarly, a {\it final} set with
respect to some orientation is a set such that no edge is directed from a vertex in the
set to a vertex not in the set. A {\it  source} is a vertex with no directed edges coming into it.

The paper \cite{NevPinUgo17} established the existence of good orientations with some special properties, but first we need an important remark, also from \cite{NevPinUgo17}.

\begin{remark}[{\cite[Rem.~4.2]{NevPinUgo17}}]
\label{rmk:orientation-composition} Let $P$ be a $d$-polytope, let $F$ be a face of $P$ and let $O$ be a good orientation of $G(P)$ in which $V(F)$ is initial. Further, denote by $O|_F$ the good orientation of $G(F)$ induced by $O$.  If $O_{F}'$ is a good orientation of $G(F)$ other than $O|_F$, then the orientation $O'$ of $G(P)$ obtained from $O$ by redirecting  the edges of $G(F)$ according to $O_{F}'$ is  also a good orientation.
\end{remark}

\begin{lemma}[{\cite[Lem.~4.3]{NevPinUgo17}}] \label{lem:Orientation-F-Initial} Let $P$ be a polytope. For every two disjoint faces $F_{i}$  and $F_{j} $ of $P$,  there is a good orientation of $G(P)$ such that
\begin{enumerate}[(i)]
\item  the vertices in $F_{i}$ are initial,
\item  the vertices in $F_{j}$ are final, and
\item  within the face $F_{i}$, any two vertices (if they exist) can be chosen to be the (local) sink and the (global) source.
\end{enumerate}
\end{lemma}
The following corollary follows from \cref{lem:Orientation-F-Initial}.

\begin{corollary}\label{cor:Orientation-F-Initial} Let $F$ be a facet of a polytope $P$. Then
\begin{enumerate}[(i)]
\item For any two vertices $u,v\in F$, there exists a good orientation $O$ of $G(P)$ such that $u$ is the source of $O$, the vertices of $F$ are an initial set, and $v$ is the sink of $O|_F$.

\item For any face $R$ in $F$, there is a good orientation $O$ of $G(P)$ such that the vertices of $R$ are an initial set in $O|_F$, and some vertex  in $F\setminus R$ is the sink  in $O|_F$.
\end{enumerate}
\end{corollary}
 \begin{proof} Let $F$ be a facet of a polytope. For the proof of (i), apply \cref{lem:Orientation-F-Initial} to $F_{i}=F$ (and disregard $F_{j}$).

For the proof of (ii), we apply \cref{lem:Orientation-F-Initial}  twice. First apply it to $F=F_{i}$, disregarding $F_{j}$,  to obtain that $V(F)$ is initial with respect to some good orientation $O$. Secondly, apply it to the polytope $F$ in $\aff F$ (disregarding $P$), where the face $R$ plays the role of $F_{i}$ and a vertex $u$ of $F$ not in $R$ plays the role of $F_{j}$; in this way, we obtain that, within $F$, the vertex set of $R$ is initial  and $u$ is a sink with respect to some good orientation $O'_{F}$ of $G(F)$. From \cref{rmk:orientation-composition} it then follows that the orientation $O'$ obtained from $O$  by directing the edges of $G(F)$ according to $O_{F}'$
is the desired good orientation.
\end{proof}

A {\it feasible}  subgraph is any induced $(d-1)$-connected subgraph $H$ of $G$ in which  the simple vertices of $P$ in  $H$ each have degree $d-1$ in $H$. In this case, each nonsimple vertex of $P$ in $H$ has degree $\ge d-1$ in $H$.

\begin{lemma}[{\cite[Lem.~4.4]{NevPinUgo17}}] \label{lem:feasible-subgraphs}  Let $P$ be a $d$-polytope, and let $H$ be a feasible subgraph of $G(P)$ containing at most $d-2$ nonsimple vertices. If the graph $G(F)$ of some facet $F$ is contained in $H$, then $H= G(F)$.
\end{lemma}

We are now ready to state and prove the main result of the section.

\begin{theorem}\label{thm:ExcessRec} Let $P$ be a $d$-polytope with excess at most $d-1$. Then the graph of $P$ determines the entire combinatorial structure of $P$.

This result is best possible in the sense that there are $d$-polytopes with excess $d$ which are not reconstructible from their graphs.
\end{theorem}
\begin{proof} Let $d\ge 4$. By \cref{lem:Joswig-valid-frames} graph reconstruction follows from determining the valid frames of each vertex.  As a result, it suffices to determine the valid frames of nonsimple vertices.

We first consider the case of $\xi=d-2$. In view of \cref{lem:excess-d-2-full-story}, a $d$-polytope with excess $d-2$ has either a unique nonsimple vertex or has $d-2$ vertices of excess degree one, which form a $(d-3)$-simplex $R$. The reconstruction of the former case follows from \cref{thm:fourpointfive}. Hence we only deal with the latter case.

The facets containing a nonsimple vertex in $R$ (and in $P$) fall into two classes: those touching but not containing $R$ and those containing $R$. By \cref{lem:excess-d-2-full-story}(iii), for each nonsimple vertex $u$ in $P$ the facet containing $u$ and missing some vertex $v$ in $R$ is given by the $d$-frame rooted at $u$ which misses $v$: for each such vertex $u$ there are exactly $d-3$ such facets.

We now deal with the facets containing $R$; here more work is required to get the valid frames of a nonsimple vertex.

Denote by $\mathcal{H}_R$ the set of feasible subgraphs which contain the complete graph $G(R)$ on the vertices of $R$, and by  $\mathcal{A}_R$  the set of all acyclic orientations of $G(P)$ in which for some subgraph $H_R$ in $\mathcal{H}_R$, (1) $H_R$  is initial and (2)  $H_R$ contains $G(R)$ as an initial subgraph. Any such $H_R$ has a sink which is a simple vertex. Observe that there is a facet $F_{R}$ in $P$ which contains $R$. The graph of $F_{R}$ is in $\mathcal{H}_R$.

\begin{claim} \label{cl:1}A feasible subgraph $H_R$  of $G(P)$  is the graph of a facet containing $R$ if and only if (1) $H_R$ contains $G(R)$,  (2) $H_R$ is initial with respect to a good orientation $O$ in $\mathcal{A}_R$, and (3) $H_R$ has a unique sink which is a simple vertex.
\end{claim}
\begin{proof}

We reason as in the proof of Claim 1 of \cite[Thm.~4.8]{NevPinUgo17}.

First consider a facet $F_R$ containing $R$. \cref{cor:Orientation-F-Initial}(ii) ensures the existence of a good orientation of $G(P)$ in which the vertices of $F_R$ are initial, that the vertices of $R$ are initial within $F_R$, and that a  simple vertex  is a sink in the facet.  This proves the ``only if'' part of the claim.

 Let $O\in \mathcal{A}_R$ and let $h_k^O$ denote the number of simple vertices of $G$ with indegree $k$. Define\[f^O_{R}:=h^O_{d-1}+dh_d^O.\]
The function $f^O_{R}$ counts the number of pairs $(F,w)$, where $F$ is a facet of $P$ and $w$ is a simple sink in $F$ of the orientation $O$ in $\mathcal{A}_R$. Since the orientation is acyclic, every facet has a sink.

Let $H_R$ be a feasible subgraph in $\mathcal{H}_R$, and let $x$ be the simple sink in $H_R$ with respect to $O$. Suppose $H_R$ does not represent the facet $F_R$ containing $x$ and the $d-1$ edges in $H_{R}$ incident to $x$. Then, in view of \cref{lem:feasible-subgraphs}, there are vertices of  $F_R$ outside $H_R$. Since $H_R$ is initial with respect to $O$, the facet $F_R$ would contain two sinks, one of them being $x$.

Consequently, given that there is a good orientation in  $ \mathcal{A}_R$ and a subgraph $H_R$ representing a facet, we have that \[\min_{O\in \mathcal{A}_R} f^O_{R}=f_{d-1},\]
where $f_{d-1}$ denotes the number of facets in $P$. Observe that all the nonsimple vertices of $P$ are in $R$, and thus, in $H_{R}$. Also, an orientation of $\mathcal{A}_R$ minimising $f_{R}^O$ must be a good orientation.

 Let $x$ be the simple sink in $H_R$ with respect to $O$, then $x$ defines a unique facet $F_R$ of $P$. Therefore, all the other vertices of $F_R$ are smaller than $x$ with respect to the ordering induced by $O$. Since $H_R$ is an initial set in $O$ and since there is directed path in $G(F_R)$ from any other vertex of $G(F_R)$ to $x$,  we must have $V(F_R) \subseteq V(H_R)$ and we are home by \cref{lem:feasible-subgraphs}.
\end{proof}
	Thanks to \cref{cl:1}, running through all the good orientations in  $\mathcal{A}_R$, we can recognise all the graphs of facets containing $R$; say that its set is $\mathcal{F}_{R}$. Consequently, for each nonsimple vertex in $R$ we also have the valid frames in each of these facets.

We now know all the valid frames of each nonsimple vertex in $P$, and of course, of each simple vertex. Thus, the case now follows from \cref{lem:Joswig-valid-frames}.

For the case of $\xi=d-1$, thanks to \cref{lem:excess-d-1-full-story} and \cref{thm:fourpointfive}, we can assume that the polytope has dimension five and the four vertices of degree six are contained in a quadrilateral 2-face $R$.

The facets containing a nonsimple vertex in $R$ (and in $P$) fall into two classes: those intersecting but not containing $R$ and those containing $R$. By virtue of \cref{lem:excess-d-1-full-story}(iii), for each nonsimple vertex $u$ in $P$ the facet containing $u$ and missing some vertex $v$ in $R$ is given by the $5$-frame rooted at $u$ which misses $v$: for each such vertex $u$ there are exactly two such facets.

To recognise the facets containing $R$ we proceed mutatis mutandis as in the case of $\xi=d-2$, just replacing the $(d-3)$-simplex with the 2-face. In this way, we recognise all the valid frames of each nonsimple vertex in $P$. Thus, the case again follows from \cref{lem:Joswig-valid-frames}.

A bipyramid over a $(d-1)$-simplex and a pyramid over a bipyramid over a $(d-2)$-simplex give a pair  of nonreconstructible $d$-polytopes with excess $d$, and exactly $d$ vertices of degree $d+1$, for $d\ge 4$. Hence the theorem is tight, as claimed.\end{proof}

\section{Acknowledgments}
Guillermo Pineda would like to thank Michael Joswig for the hospitality at the Technical University of Berlin and for suggesting looking at the reconstruction problem for polytopes with small number of vertices.


\end{document}